\documentclass[a4paper,11pt]{article}
\usepackage{amsmath,amsfonts,enumerate,amssymb,amsthm,color}
\usepackage[english]{babel}
\usepackage{pgf,tikz}

\title{Automorphism groups of Cayley graphs
generated by block transpositions and regular Cayley maps}

\author{Annachiara Korchmaros,$^{a}$ \;
Istv\'an Kov\'acs$^{\, b}$ \\  [+0.75ex]
$^a$  {\small School of Mathematics Georgia Institute of
Technology,}  \\ [-0.5ex]
{\small 686 Cherry Street Atlanta, GA 30332-0160 USA} \\ [-0.5ex]
$^b$ {\small IAM, University of Primorska, Muzejski trg 2,
SI-6000 Koper, Slovenia}
}
\date{}

\newtheorem{theorem}{Theorem}
\newtheorem{lemma}{Lemma}
\newtheorem{corollary}{Corollary}
\newtheorem{proposition}{Proposition}
\theoremstyle{definition}
\newtheorem{example}{Example}
\theoremstyle{remark}
\newtheorem{remark}{Remark}

\newcommand{\Cay}{{\rm{Cay(Sym_n}},T_n)}
\newcommand{\Aut}{\mbox{\rm Aut}}
\newcommand{\Sym}{\mbox{\rm Sym}}
\newcommand{\cupdot}{\mathbin{\mathaccent\cdot\cup}}
\newcommand{\UCay}{{\rm{Cay(Sym_n}},U)}
\newcommand{\Alt}{\mbox{\rm Alt}}
\newcommand\Z{\mathbb{Z}}
\newcommand{\CM}{\mbox{\rm CM}}
\newcommand{\Cayg}{\mbox{\rm Cay}}
\newcommand\M{\mathcal{M}}
\newcommand\D{\mathsf{D}}
\newcommand\F{\mathsf{F}}
\newcommand\f{\mathsf{f}}
\newcommand\g{\mathsf{g}}
\begin{document}

\maketitle
\let\thefootnote\relax\footnote{
The second author was partially supported by the
Slovenian Research Agency (research program P1-0285 and research projects N1-0032, N1-0038, J1-5433, and J1-6720). \\  [+0.5ex]
{\it  E-mail addresses:}
akorchmaros3@gatech.edu (A.~Korchmaros),
istvan.kovacs@upr.si (I.~Kov\'acs).
}

\begin{abstract}
This paper deals with the Cayley graph $\Cay,$ where the generating set consists of all block transpositions. A motivation for the study of these particular Cayley graphs comes from current research in Bioinformatics. As the main result, we prove that
$\Aut(\Cay)$ is the product of the left
translation group by a dihedral group $\D_{n+1}$ of order $2(n+1)$. The proof uses several properties of the subgraph $\Gamma$
of $\Cay$ induced by the set $T_n$. In particular,
$\Gamma$ is a $2(n-2)$-regular graph whose automorphism group is
$\D_{n+1},$ $\Gamma$ has as many as $n+1$ maximal cliques of size $2,$ and its subgraph $\Gamma(V)$ whose vertices are those in these cliques is a $3$-regular, Hamiltonian, and
vertex-transitive graph. A relation of the
unique cyclic subgroup of $\D_{n+1}$ of order $n+1$ with
regular Cayley maps on $\Sym_n$ is also discussed. It is shown
that the product of the left translation group by the latter group can be obtained as the automorphism group of a
non-$t$-balanced regular Cayley map on $\Sym_n$.
\medskip

\medskip


\end{abstract}

\section{Introduction}
\label{intro}

Block transpositions are well-known sorting operations with relevant applications in Bioinformatics, see~\cite{FLRTV}. They act on a string by removing a block of consecutive entries and inserting it somewhere else. In terms of the symmetric group $\Sym_n$ of degree $n$, the strings are identified with the permutations on $[n]=\{1,2,\ldots n\}$, and block transpositions are defined as follows. For any three integers $i,j,k$ with $0\le i <j<k \le n$, the \emph{block transposition} $\sigma(i,j,k)$ with cut points $(i,j,k)$ turns the permutation  $\pi=[\pi_1\,\cdots\,\pi_n]$
into the permutation $\pi'=[\pi_1\cdots \pi_i\,\,\pi_{j+1}\cdots \pi_k\,\,\pi_{i+1}\cdots \pi_{j}\,\,\pi_{k+1}\cdots \pi_n]$. This action of $\sigma(i,j,k)$ on $\pi$ can also be expressed as the composition $\pi'=\pi\circ\sigma(i,j,k)$.
The set $T_n$ of all block transpositions has size $(n+1)n(n-1)/6$ and
is an inverse closed generating set of $\Sym_n.$ The arising Cayley graph $\Cay$ is a very useful tool since ``sorting a permutation by block transpositions'' is equivalent to finding shortest paths between vertices in $\Cay,$ see~\cite{EEKSW,FLRTV,MS}.

Although the definition of a block transposition arose from a practical need, the Cayley graph $\Cay$ also presents some
interesting theoretical features (cf.~\cite{K}); and the most remarkable one is perhaps the existence of automorphisms other than left translations. We remark that the automorphism
groups of certain Cayley graphs on Symmetric groups have attracted some attention recently, see \cite{DZ,F,G}.
\medskip

In this paper, we focus on the automorphisms of the Cayley graph
$\Cay,$ and denote its full automorphism group by $\Aut(\Cay)$.
In particular, we introduce a subgroup of automorphisms contained in the vertex stabilizer of the
identity permutation $\iota,$ which is isomorphic to the dihedral group of order $2(n+1)$.  We denote this dihedral group by $\D_{n+1},$
and since it arises from the toric equivalence in $\Sym_n$ and the reverse permutation, we name it the \emph{toric-reverse group}.

In addition, the \emph{block transposition graph}, that is the subgraph
$\Gamma$ of $\Cay$ induced by the vertex set $T_n$
 has especially nice properties. As we show in this paper, $\Gamma$ is a $2(n-2)$-regular graph whose automorphism group
coincides with the permutation group induced by toric-reverse group $\D_{n+1}$. Clearly, as every automorphism in
$\D_{n+1}$ fixes $\iota,$ it also preserves the neighborhood of $\iota$
in $\Cay,$ which is the set $T_n$. Therefore, $\D_{n+1}$ induces an automorphism group of $\Gamma$.
We show that the latter group is, in fact, equal to $\Aut(\Gamma),$
the full automorphism group of $\Gamma$.

\begin{theorem}\label{main1}
The full automorphism group of $\Gamma$ is equal to the permutation group induced by toric-reverse group $\D_{n+1}$.
\end{theorem}

The key part of the proof of Theorem~\ref{main1} is a careful 
analysis of the maximal cliques of $\Gamma$ of size $2$.
We show that $\Gamma$ has precisely $n+1$ such cliques and also look inside the subgraph $\Gamma(V)$ of $\Gamma$ induced by the set $V$ whose vertices are the $2(n+1)$ vertices of these cliques. We prove that $\Gamma(V)$ is $3$-regular, and $\D_{n+1}$
induces an automorphism group of $\Gamma(V)$ acting transitively (and hence regularly) on $V.$ As a curiosity, we also observe that
$\Gamma(V)$ is Hamiltonian. This agrees with a result of Alspach and Zhang~\cite{AZ} and confirms the Lov\'asz conjecture \cite{L} for $\Gamma(V).$

Then, we turn to the automorphism group $\Aut(\Cay),$ and using Theorem~\ref{main1}, we show that any automorphism
fixing $\iota$ belongs to $\D_{n+1}$. 
Therefore, $\Aut(\Cay)$ will be completely determined.

\begin{theorem}\label{main2}
The full automorphism group of $\Cay$ is the product of the left translation group by the toric-reverse group
$\D_{n+1}$.
\end{theorem}

Finally, we also discuss a relation of the \emph{toric group}
$\bar{\F},$ the unique cyclic subgroup of $\D_{n+1}$ of order $n+1,$ with regular Cayley maps defined on the symmetric group
$\Sym_n$. We show that if $n\ge 5,$ then
the product of the left translation group by the toric group $\bar{\F}$
can be characterized as the automorphism group of a
non-$t$-balanced regular Cayley map on $\Sym_n$ (see
Theorem~\ref{thm11sept}). This follows from the observation that any element in $\bar{\F}$ is a skew-morphism of the group $\Sym_n$ in the sense of \cite{JS}.

\section{Background on block transpositions}
\label{sec:2}

In this paper, all groups are finite, and all graphs are finite and simple. For basic facts on Cayley graphs and combinatorial properties of permutations the reader is referred to~\cite{AB,Ba}.

Throughout the paper, $n$ denotes a positive integer. In our investigation cases $n\leq 3$ are trivial while case $n=4$ presents some results different from the general case.

For a set $X$ of size $n,$ $\Sym_X$ stands for the set of all permutations on $X$. For the sake of simplicity, $[n]=\{1,2,\ldots,n\}$ is usually taken for $X.$ As it is customary in the literature on block transpositions, we mostly adopt the functional notation for permutations:  If $\pi\in \Sym_n,$ then $\pi=[\pi_1\pi_2\cdots \pi_n]$ with $\pi(t)=\pi_t$ for every $t\in[n],$ and if
$\pi,\rho\in \Sym_n$ then $\tau=\pi\circ\rho$ is the permutation defined by $\tau(t)=\pi(\rho(t))$ for every $t\in [n]$. The \emph{reverse permutation} is $w=[n\,n-1\cdots 1]$, and $\iota=[1\,2\cdots n]$ is the \emph{identity permutation}.

For any three integers, named \emph{cut points}, $(i,j,k)$ with $0\leq i< j< k\leq n$, the \emph{block transposition} (transposition, see~\cite{GBH}) $\sigma(i,j,k)$ is defined to be the function on $[n]$:
\begin{equation}\label{feb9}
\sigma(i,j,k)_t =\left\{\begin{array}{ll}
t, &  1\leq t\leq i\quad k+1\leq t\leq n,\\
t+j-i, & i+1\leq t\leq k-j+i,\\
t+j-k, & k-j+i+1\leq t\leq k.
\end{array}
\right.
\end{equation}This shows that $\sigma(i,j,k)_{t+1}=\sigma(i,j,k)_{t} +1$ in the intervals:
\begin{equation}
\label{function}
[1,i],\quad[i+1, k-j+i],\quad[k-j+i+1,k],\quad[k+1, n],
\end{equation}where
\begin{align}
\label{cuppoints}
\sigma(i,j,k)_i&=i;&\sigma(i,j,k)_{i+1}&=j+1;&\sigma(i,j,k)_{k-j+i}&=k;\\
\sigma(i,j,k)_{k-j+i+1}&=i+1;&\sigma(i,j,k)_{k}&=j;&\sigma(i,j,k)_{k+1}&=k+1.\notag
\end{align}Actually, $\sigma(i,j,k)$ can also be represented as the permutation
\begin{equation}
\label{eq22ott12}
\sigma(i,j,k)=\left\{\begin{array}{ll}
[1\cdots i\,\, j+1\cdots k\,\, i+1\cdots j\,\, k+1 \cdots n], & 1\leq i,\,k< n,\\
{[j+1\cdots k\,\, 1\cdots j\,\, k+1 \cdots n]}, & i=0,\,k< n,\\
{[1\cdots i\,\,j+1\cdots n\,\, i+1\cdots j]}, & 1\leq i,\,k=n,\\
{[j+1\cdots n\,\, 1\cdots j]}, & i=0,\,k=n
\end{array}
\right.
\end{equation}such that the action of $\sigma(i,j,k)$ on $\pi$ is defined as the product
$$\pi\circ\sigma(i,j,k)=[\pi_1\cdots \pi_i\,\,\pi_{j+1}\cdots \pi_k\,\,\pi_{i+1}\cdots \pi_{j}\,\,\pi_{k+1}\cdots \pi_n].$$ Therefore, applying a block transposition on the right of $\pi$ consists in switching two adjacent subsequences of $\pi$, namely \emph{blocks}, without changing the order of the integers within each block. This may also be expressed by $$[\pi_1\cdots\pi_i|\pi_{i+1}\cdots\pi_j|\pi_{j+1}\cdots\pi_k|\pi_{k+1}\cdots\pi_n].$$

{}From now on, $T_n$ denotes the set of all block transpositions on $[n].$ The size of $T_n$ is equal to $n(n+1)(n-1)/6.$ Obviously, $T_n$ is not a subgroup of $Sym_n.$ Nevertheless,
$T_n$ is power and inverse closed. For any cut points $(i,j,k),$
\begin{equation}
\label{eqa18oct}
\sigma(i,j,k)^{-1}=\sigma(i,k-j+i,k),\qquad\sigma(i,i+1,k)^{j-i}=\sigma(i,j,k).
\end{equation}Also, for any two integers $i,\,k$ with $0\le i<k\le n$ the subgroup generated by $\sigma(i,i+1,k)$ consists of all $\sigma(i,j,k)$ together with the identity. In particular, $\sigma(0,1,n)$ generates a subgroup of order $n$ that often appears in our arguments. Throughout the paper, $\beta=\sigma(0,1,n)$ and $B$ denotes the set of nontrivial elements of the subgroup generated by $\beta$.

We introduce some subsets in $T_n$ that play a relevant role in our study. Every permutation $\bar{\pi}$ on $[n-1]$ extends to a permutation $\pi$ on $[n]$ such that $\pi_t=\bar{\pi}_t$ for $1\le t \le n-1$ and $\pi_n=n.$ Hence,
$T_{n-1}$ is naturally embedded in $T_n$ since every $\sigma(i,j,k)\in T_n$ with $k\neq n$ is identified with the block transposition $\bar\sigma(i,j,k)$. On the other side, every permutation $\pi'$ on $\{2,3,\ldots,n\}$ extends to a permutation on $[n]$ such that $\pi_t=\pi'_t,$ for $2\le t \le n$ and $\pi_1=1.$ Thus, $\sigma(i,j,k)\in T_n$ with $i\neq 0$ is identified with the block transposition $\sigma'(i,j,k).$ The latter block transpositions form the set
$$ S_{n-1}^\triangledown=\{ \sigma(i,j,k)|\,i\neq 0\}.$$ Also, $$S_{n-2}^\vartriangle=T_{n-1}\cap S_{n-1}^\triangledown$$ is the set of all block transpositions on the set $\{2,3,\ldots,n-1\}$.
Our discussion leads to the following results.

\begin{lemma}[Partition lemma]
Let $L=T_{n-1}\setminus S_{n-2}^\vartriangle$ and let $F=S_{n-1}^\triangledown\setminus S_{n-2}^\vartriangle$. Then
\label{lem1oct9}
$$T_n=B\cupdot L \cupdot F \cupdot S_{n-2}^\vartriangle.$$
\end{lemma}

With the above notation, $L$ is the set of all $\sigma(0,j,k)$ with $k\neq n$, and $F$ is the set of all $\sigma(i,j,n)$ with $i\neq0.$ Furthermore, $|B|=n-1,\,|L|=|F|=(n-1)(n-2)/2,$ and $|S_{n-2}^\vartriangle|=(n-1)(n-2)(n-3)/6.$

\section{Toric equivalence in the symmetric group}
\label{sec:3}

The definition of toric (equivalence) classes in $\Sym_n$ requires to consider permutations on $[n]^0=\{0,1,\ldots,n\}$ and recover the permutations $\pi=[\pi_1\,\cdots\pi_n]$ on $[n]$ in the form $[0\,\pi]$, where $[0\,\pi]$ stands for the permutation $[\pi_0\,\pi_1\,\cdots \pi_n]$ on $[n]^0$ with $\pi_0=0.$ Let $$\alpha=[1\,2\,\ldots\, n\,0].$$ For any integer $r$ with $0\le r \le n,$
\begin{equation}\label{alpha_powers}
 \alpha^r_x\equiv x+r {\pmod{n+1}},\qquad 0\leq x\leq n.
\end{equation}This gives rise to the
\emph{toric maps} $\f_r$ on $\rm{Sym_n}$ with $0\le r\le n,$ defined by
\begin{equation}\label{eq2oct9}
\f_r(\pi)=\rho\Longleftrightarrow[0\rho]= \alpha^{n+1-\pi_r}\circ [0\,\pi] \circ \alpha^r.
\end{equation}
The \emph{toric class} of $\pi$ is
\begin{equation}\label{eq1oct10}
\F(\pi)=\{\f_r(\pi)\,\mid\,r=0,1,\ldots,n\}.
\end{equation}
Since
\begin{equation}
\label{eq9oct}
(\f_r(\pi))_t=\pi_{r+t}-\pi_r,\qquad t\in [n],
\end{equation}where the indices are taken mod$(n+1),$ \eqref{eq1oct10} formalizes the intuitive definition of toric classes introduced in~\cite{EEKSW} by Eriksson et al.

{}From \eqref{eq2oct9}, $\f_s  \circ \f_r=\f_{s+r}$, where the indices are taken $\pmod{n+1}$.
With a slight abuse of notation, the symbol $\circ$ will also be used for the group operation of $\Sym_X$ where $X=\Sym_n$.
In fact, $(\f_s\circ\f_r )(\pi)=\f_s(\varphi)$ with $[0\,\varphi]=\alpha^{n+1-\pi_r}\circ[0\,\pi]\circ\alpha^r.$
Also, $(\f_s\circ\f_r)(\pi)=\mu$ with
$$\begin{array}{lll}
[0\,\mu]&=& \alpha^{n+1-\varphi_s}\circ [0\,\varphi] \circ \alpha^s\\
{}&=&\alpha^{n+1-\pi_{r+s}}\circ [0\,\pi] \circ\alpha^{r+s},
\end{array}$$where $\varphi_s=\pi_{r+s}-\pi_r$. Hence, $\f_r=
\f^{\,r}$ with $\f=\f_1$, and the set
$$\F=\{\f_r\,\mid\,r=0,1,\ldots, n\}$$ is a cyclic group of order $n+1$ generated by $\f.$
Consequently, the toric class of $\pi$ 
consists of the $\F$-orbit containing
$\pi,$ and being so, the number of elements in a toric class is always a divisor of $n+1$. We note that there are exactly $\varphi(n+1)$ classes that have only one element, where $\varphi$ is the Euler function,
see~\cite{C}.

If $\pi\in \Sym_n$ and $0\le r \le n,$ then
\begin{equation}
\label{lem3oct11}
(\f_{r}(\pi))^{-1}=\f_{\pi_r}(\pi^{-1}).
\end{equation}
In particular, $(\f_{r}(\pi))^{-1}=\f_{r}(\pi^{-1})$ provided that
$\pi_r=r.$

The \emph{reverse map} $\g$ on $\Sym_n$ is defined by
\begin{equation}
\label{eq2oct9b}
\g(\pi)=\rho\Longleftrightarrow[0\,\rho]=[0\,w]\circ [0\,\pi] \circ [0\,w].
\end{equation}
Thus $\g$ is an involution, and
\begin{equation}
\label{eq9octb}
(\g(\pi))_t=n+1-\pi_{n+1-t},\qquad t\in [n].
\end{equation}
Also, for every integer $r$ with $0\le r\le n,$
\begin{equation}
\label{eqoct15a}
\g\circ\f_r\circ\g=\f_{n+1-r}.
\end{equation} In fact, $(\g\circ\f_r\circ\g)(\pi)=(\g\circ\f_r)(\rho)=\g(\rho')$ with
$[0\,\rho]=[0\,w]\circ[0\,\pi]\circ[0\,w],$ and $[0\,\rho']=\alpha^{n+1-\rho_r}\circ [0\,\rho] \circ \alpha^r.$ Also, $(\g\circ\f_r\circ\g)(\pi)=
\mu$ with
$$\begin{array}{lll}
[0\,\mu]&=& [0\,w]\circ [0\,\rho'] \circ [0\,w]\\
{}&=&[0\, w]\circ\alpha^{n+1-\rho_{r}}\circ [0\, w]\circ [0\,\pi]\circ [0\, w] \circ\alpha^{r}\circ[0\, w]\\
{}&=&\alpha^{n+1-\pi_{n+1-r}}\circ [0\,\pi]\circ\alpha^{n+1-r}
\end{array}$$since $[0\,w]\circ\alpha^r \circ [0\,w]=\alpha^{n+1-r}.$

Now, we show that the toric maps and the reverse map take any block transposition into a block transposition, as corollary of the following lemma.

\begin{lemma}
\label{lemc19ott2013}
Let $\sigma(i,j,k)$ be any block transposition on $[n]$. Then
\begin{equation}
\label{22march2015A}
\f(\sigma(i,j,k))=\left\{\begin{array}{ll}
\sigma(i-1,j-1,k-1), & i >0 ,\\
\sigma(k-j-1,n-j,n), & i=0,
\end{array}
\right.
\end{equation}and
\begin{equation}\label{eq1oct11}
\g(\sigma(i,j,k))=\sigma(n-k,n-j,n-i).
\end{equation}
\end{lemma}
\begin{proof}
For $i>0,$ let $\sigma=\sigma(i,j,k)$. Then \eqref{eq2oct9} reads
$$
\f(\sigma)=\rho \Longleftrightarrow [0\,\rho]= \alpha^{n}\circ [0\,\sigma] \circ \alpha,
$$
as $\sigma_1=1$ from \eqref{eq22ott12}.
Since $\alpha_t=t+1\pmod{n+1}$ for every $t\in [n]^0$, then, by \eqref{feb9},
\begin{equation}\label{feb91}
([0\,\sigma] \circ\alpha)_t =\left\{\begin{array}{ll}
t+1, &  0\leq t+1\leq i\quad k+1\leq t +1 \leq n,\\
t+j-i+1, & i+1\leq t+1\leq k-j+i,\\
t+j-k +1, & k-j+i+1\leq t+1\leq k.\\
\end{array}
\right.
\end{equation}

In addition, $\alpha^n_t=\alpha^{-1}_t=t-1\pmod{n+1}$ for every $t\in [n]^0.$ Thus, \eqref{feb91} yields
$$
(\alpha^n\circ[0\,\sigma]\circ\alpha)_t =\left\{\begin{array}{ll}
t, &  0\leq t\leq i-1\quad k\leq t\leq n,\\
t+j-i, & i\leq t\leq k-j+i-1,\\
t+j-k, & k-j+i\leq t\leq k-1,\\
\end{array}
\right.
$$
hence the statement holds true for $i>0$, by \eqref{feb9}. Now, suppose $i=0$. Then, from \eqref{eq22ott12} it follows $\sigma_1=j+1$, and \eqref{eq2oct9} reads
$$
\f(\sigma(0,j,k))=\rho \Longleftrightarrow [0\rho]= \alpha^{n-j}\circ [0\,\sigma(0,j,k)] \circ \alpha.
$$
Replacing $i$ with $0$ in \eqref{feb91},
$$(\alpha^{n-j}\circ [0\,\sigma(0,j,k)]\circ\alpha)_t =
\left\{\begin{array}{ll}
t, &  0\leq t\leq k-j-1,\\
t-k, & k-j\leq t\leq k-1,\\
t-j, & k\leq t\leq n.\\
\end{array}
\right.$$
Therefore, the assertion for $\f
(\sigma(0,j,k))$ follows from \eqref{feb9}.

For the reverse map, \eqref{eq22ott12} yields \eqref{eq1oct11}.
\end{proof}

Since the toric maps are powers of $\f$,
Lemma~\ref{lemc19ott2013} has the following corollary.

\begin{corollary}
\label{th1}
$T_n$ is invariant under the action of the toric maps and the reverse map.
\end{corollary}

\section{The toric-reverse group $\boldsymbol{\D_{n+1}}$}

\label{sec:4}

Since $T_n$ is an inverse closed generator set of $\Sym_n$ which does not contain the identity permutation $\iota$, the Cayley graph $\Cay$ is an undirected connected simple graph, where $\{\pi,\rho\}$ is an edge if and only if $\rho=\pi\circ\sigma(i,j,k),$ for some $\sigma(i,j,k)\in T_n.$

By a result of Cayley, every $h\in\Sym_n$ defines
\emph{the left translation}, denoted by 
$L_h,$ which is the automorphism of $\Cay$ that maps the vertex
$\pi$ to the vertex $h\circ\pi$, and hence the edge $\{\pi,\rho\}$ to the edge $\{h\circ\pi,h\circ\rho\}.$
These automorphisms form the
\emph{left translation group} $L(\Sym_n)$.

Clearly, $\Sym_n\cong L(\Sym_n)$.
Furthermore, since $L(\Sym_n)$
acts regularly on $\Sym_n,$ every automorphism of $\Cay$ is the product of a left
translation by an automorphism fixing $\iota.$

We point out that the automorphism group $\Aut(\Cay)$ contains several nontrivial elements fixing $\iota.$ The reverse map $\g$ is one of them. In fact, \eqref{eq2oct9b} yields
\begin{equation}
\label{eq16aug2015}
\g(\rho\circ\pi)=\g(\rho)\circ\g(\pi),
\end{equation}
 whence if $\rho=\pi\circ\sigma(i,j,k)$ then $\g(\rho)=\g(\pi)\circ\g(\sigma(i,j,k))$ with $\g(\sigma(i,j,k))\in T_n$ by \eqref{eq1oct11}.

Further such automorphisms arise from the toric maps. To show this we need some results about the map
 $\bar{\f}$ defined by
 $$\bar{\f}(\pi)=(\f(\pi^{-1}))^{-1}\qquad \pi\in \Sym_n,$$
and more generally on the map $\bar{\f}_r$ be defined by
\begin{equation}
\label{March29+}
\bar{\f}_r(\pi)=(\f_r(\pi^{-1}))^{-1},\qquad \pi\in \Sym_n.
\end{equation}
In other words,
\begin{equation}\label{eq2jul27}
\bar{\f}_r(\pi)=\rho\Longleftrightarrow[0\rho]= \alpha^{n+1-r}\circ [0\,\pi] \circ \alpha^{(\pi^{-1})_r}.
\end{equation}Then $\bar{\f}_r=\bar{\f}^{\,r}$, as $\f_r=\f^{\,r}$ for any integer $r$ with $0\leq r\leq n$. Hence $\bar{\F}\cong\F,$ where
$\bar{\F}$ is the group generated by $\bar{\f},$ and the natural
mapping $\bar{\f}_{r} \mapsto\f_{r}$ is an isomorphism.

\begin{lemma}
\label{22marchC2015}
Let $\sigma(i,j,k)$ be any block transposition on $[n].$ Then
\begin{equation}
\label{22march2015}
\bar{\f}(\sigma(i,j,k))=\left\{\begin{array}{ll}
\sigma(i-1,j-1,k-1), & i >0 ,\\
\sigma(j-1,k-1,n), & i=0.
\end{array}
\right.
\end{equation}
\end{lemma}
\begin{proof}
Let $\sigma=\sigma(i,j,k).$ For $i>0,$ we obtain $\sigma_1=1$ from \eqref{eq22ott12}. Therefore, $\bar{\f}(\sigma)=
\f(\sigma)$
by  \eqref{lem3oct11} and \eqref{March29+}. Hence the statement for $i>0$ follows from Lemma~\ref{lemc19ott2013}.

Now, suppose $i=0.$ By \eqref{March29+} and Lemma~\ref{lemc19ott2013},
$$\bar{\f}(\sigma)=(\f(\sigma^{-1}))^{-1}=(\f(\sigma(0,k-j,k)))^{-1}=\sigma(j-1,n-(k-j),n)^{-1}$$
which is equal to $\sigma(j-1,k-1,n),$ by \eqref{eqa18oct}. Therefore, the statement also holds for $i=0.$
\end{proof}

Since $\bar{\f}_r$ is a power of $\bar{\f}$ for every $2\leq r\leq n$,
Lemma~\ref{22marchC2015} has the following corollary.

\begin{corollary}
\label{cor1July29}
$T_n$ is invariant under the action of $\bar{\F}.$
\end{corollary}

\begin{lemma}\label{mat27oct}
For every $\pi,\rho \in {\Sym}_n,$
$$\bar{\f}_r(\rho\circ\pi)=\bar{\f}_r(\rho)\circ\bar{\f}_s(\pi),$$
where $s=(\rho^{-1})_r;$
\end{lemma}
\begin{proof}
From \eqref{eq2jul27}, $\bar{\f}_r(\rho\circ\pi)=\mu$ with
$t=\pi^{-1}\circ\rho^{-1}$ and
$$\begin{array}{lll}
[0\,\mu]&=& \alpha^{n+1-r}\circ [0\,\rho] \circ [0\,\pi]\circ \alpha^{t_r}\\
{}&=&\alpha^{n+1-r}\circ [0\,\rho]\circ\alpha^{(\rho^{-1})_r} \circ\alpha^{n+1-(\rho^{-1})_r}\circ [0\,\pi]\circ \alpha^{t_r}.
\end{array}$$ The  assertion follows from \eqref{eq2jul27}.
\end{proof}

Now, we transfer our terminology from Sect.~\ref{sec:3}. In particular,
$\bar{\f}$ and its powers are the \emph{toric maps}, and
$\bar{\F}$ is the \emph{toric group}.

\begin{proposition}\label{prop10oct}
The toric maps and the reverse map are in the automorphism group\\
$\Aut(\Cay)$.
\end{proposition}
\begin{proof}It has already been shown  that
the reverse map $\g$ is in $\Aut(\Cay),$ see \eqref{eq16aug2015}. 
Here, we deal with the toric maps.
Let $\pi,\rho \in {\Sym}_n$ be any two adjacent vertices of $\Cay.$ Then, $\rho=\pi\circ\sigma$ for some $\sigma=\sigma(i,j,k)\in T_n.$ From this $\bar{\f}(\rho)=\bar{\f}(\pi)\circ\bar{\f}_{(\pi^{-1})_1}(\sigma)$ by Lemma~\ref{mat27oct}. Therefore, the assertion for $\bar{\f}$ follows from Corollary~\ref{cor1July29}.
This implies that every toric map is in $\Aut(\Cay)$.
\end{proof}

From (\ref{eqoct15a}), for every integer $r$ with $0\le r\le n,$
\begin{equation}
\label{eq1July30}
\g\circ\bar{\f}_r\circ\g=\bar{\f}_{n+1-r}.
\end{equation}
Therefore, the set consisting of $\bar{\F}$ and its coset
$\bar{\F}\circ\g$ is a dihedral group of order $2(n+1).$
Clearly, every element of this group fixes $\iota.$ We denote this dihedral group by
$\D_{n+1}$, and from now on, the term of \emph{toric-reverse group} stands for $\textsf{D}_{n+1}$.
Now, Proposition~\ref{prop10oct} has the following corollary.

\begin{corollary}\label{teor1}
The automorphism group of $\Cay$ contains a dihedral subgroup of order $2(n+1)$ fixing the identity permutation.
\end{corollary}

We conclude this section by showing that the product
$L(\Sym_n)\D_{n+1}$ is in fact a group, which is
isomorphic to the direct product $\Sym_{n+1} \times \Z_2$.

\begin{proposition}\label{29oct}
The product of the left
translation group by the
toric-reverse group is a group 
isomorphic to the direct product of ${\Sym}_{n+1}$ by a group of order $2.$
\end{proposition}
\begin{proof} First, we find an involution $t$ contained in
$L(\Sym_n)\D_{n+1}$ that centralizes
$L(\Sym_n)
\bar{\F}.$ Two automorphisms of $\Cay$ arise from the reverse permutation, namely $\g$ and the  left
translation $L_w$. We choose $t$ to be the
automorphism $L_w\circ\g$, and hence
$t$ takes $\pi$ to $\pi\circ w.$
Obviously, $t \in L(\Sym_n)\D_{n+1},$ and it is an involution as
$\g$ and $L_w$ are commuting
involutions.

Now, we show that $t$ centralizes $L(\Sym_n)
\bar{\F}$. In order to do that, it suffices to prove $t\circ\bar{\f}=\bar{\f}\circ t$ since $t$ commutes with any left
translation. Now, for every $\pi\in \Sym_n,$
$$(\bar{\f}\circ t)(\pi)=\rho\Longleftrightarrow [0\,\rho]=\alpha^n\circ[0\,\pi]\circ[0\,w]\circ\alpha^{s},$$where
$s=[(\pi\circ w)^{-1}]_1= n+1-(\pi^{-1})_1$ by \eqref{eq2jul27}. On the other hand,
$$(t \circ\bar{\f})(\pi)=\rho' \Longleftrightarrow [0\,\rho']=\alpha^n\circ[0\,\pi]\circ\alpha^{(\pi^{-1})_1}
\circ[0\, w].$$Since $[0\, w]\circ\alpha^{n+1-(\pi^{-1})_{1}}=\alpha^{(\pi^{-1})_1}\circ[0\, w]$ by \eqref{eqoct15a}, this yields
that $t$ commutes with $\bar{\F}$.

Now, we show that $t$ is not contained in
$L(\Sym_n)
\bar{\F}.$ Suppose on the contrary that there exists some left
translation $L_h$ such that $t=L_h\circ\bar{\f}^{\,r}$ with
$0\le r \le n.$
Then it follows  $w=\iota\circ w=t(\iota)=
(L_h\circ\bar{\f}^{\,r})(\iota)=L_h(\iota)=h\circ\iota=h$.
Therefore, $t=L_w\circ\bar{\f}^{\,r},$ and hence
$\g=L_w\circ t=\bar{\f}^{\,r}$.
This implies that $\g$ commutes with every toric
map in $\bar{\F},$ which is in contradiction with  \eqref{eq1July30}.

By Lemma~\ref{mat27oct}, for every $h,\pi\in\Sym_n,$
$(\bar{\f}_r\circ L_h)(\pi)=\bar{\f}_{r}(h\circ\pi)=(L_{h'}\circ \bar{\f}_{s})(\pi),$
where $h'=\bar{\f}_{r}(h)$ and $s=(h^{-1})_r$. This yields
$\bar{\F}L(\Sym_n) \subseteq L(\Sym_n)\bar{\F},$  and thus the product $L(\Sym_n)\bar{\F}$ is indeed a group.
As $\g\circ t^{-1}= L_w\in L(\Sym_n)\bar{\F}$, the coset
$(L(\Sym_n)\bar{\F})\circ t$ is equal to the coset $(L(\Sym_n)\bar{\F})\circ\g$. Therefore,
$L(\Sym_n)\D_{n+1}=L(\Sym_n)(\bar{\F}
\langle\,\g\,\rangle)=(L(\Sym_n)\bar{\F})
\langle\,\g\,\rangle=(L(\Sym_n)\bar{\F})\langle\, t\,\rangle=
L(\Sym_n)\bar{\F}\times\langle\, t\,\rangle$.

In fact, a straightforward computation shows that
$$L_h\circ\bar{\f}^{-r}=L_{h'}\circ\bar{\f}^{\,r}\circ L_w,$$where
$h=h'\circ w,$ for any right translation $L_h$ and $0\le r\le n.$

To prove the isomorphism $L(\Sym_n)
\bar{\F} \cong \Sym_{n+1},$ let $\Phi$ be the map that takes
$L_h\circ\bar{\f}^{\,r}$ to $[0\,h]\circ\alpha^{n+1-r}.$
For any $h,k,\pi\in\Sym_n$ and $0\leq r,u \le n$, by Lemma~\ref{mat27oct}, 
$$
(L_h\circ\bar{\f}_r\circ L_k\circ\bar{\f}_u)(\pi)=
(L_h\circ\bar{\f}_r)(k\circ\bar{\f}_u(\pi))=
h\circ\bar{\f}_{r}(k)\circ\bar{\f}_{u+(k^{-1})_r}(\pi).
$$
This shows that $L_h\circ\bar{\f}_r\circ L_k\circ\bar{\f}_u=L_d\circ\bar{\f}_{u+(k^{-1})_r}$ with $d=h\circ\bar{\f}_r(k)$ and
$L_d\in  L(\Sym_n).$

Then,
$$\begin{array}{lll}
\Phi(L_h\circ\bar{\f}_r\circ L_k\circ\bar{\f}_u)&=&[0\,h]\circ[0\,\bar{\f}_r(k)]\circ\alpha^{n+1-u-(k^{-1})_r}\\
{}&=&[0\,h]\circ\alpha^{n+1-r}\circ[0\,k]\circ\alpha^{n+1-u}.
\end{array}$$On the other hand,
$$\Phi(L_h\circ\bar{\f}_r)\circ\Phi(L_k\circ\bar{\f}_u)=[0\,h]\circ\alpha^{n+1-r}\circ[0\,k]\circ\alpha^{n+1-u}.$$Hence, $\Phi$ is a group homomorphism from
$L(\Sym_n)
\bar{\F}$ into the symmetric group on $[n]^0$. Furthermore,
$\ker(\Phi)$ is trivial. In fact, $[0\,h]\circ \alpha^{n+1-r}=[0\,\iota]$ only occurs for $h=\iota$ and $r=0$ since the inverse of $\alpha^{n+1-r}$ is the permutation $\alpha^{r}$ not fixing $0$
whenever $r \ne 0$. This together with
$| L(\Sym_n)
\bar{\F}|=(n+1)!$ shows that $\Phi$ is bijective.
\end{proof}

\section{Combinatorial properties of the block transposition graph
$\boldsymbol{\Gamma}$}
\label{sec:5}

As mentioned in the introduction, a crucial step
to find $\Aut(\Cay)$ is 
the subgraph of $\Cay$ induced by the set $T_n$. We call this subgraph
the \emph{block transposition graph,} and denote it by $\Gamma$.
In this section we derive several combinatorial properties of the latter
graph.

We begin with some results on the components of the partition in Lemma~\ref{lem1oct9}.
Since $B$ consists of all nontrivial elements of a subgroup of $T_n$ of order $n,$ the block transpositions in $B$ are the vertices of a complete graph  of size $n-1.$ Lemma~\ref{lem1oct9} and \eqref{eq1oct11} give the following property.

\begin{corollary}\label{lem2oct11}
The reverse map preserves both $B$ and $S_{n-2}^\vartriangle$ while it switches $L$ and $F$.
\end{corollary}

\begin{lemma}\label{lem1oct11}
No edge of $\Cay$ has one endpoint in $B$ and the other in $S_{n-2}^\vartriangle$.
\end{lemma}
\begin{proof}
Suppose on the contrary that $\{\sigma(i',j',k'),\sigma(0,j,n)\}$ with $i'\neq 0$ and $k'\neq n$ is an edge of $\Cay.$ By \eqref{eqa18oct},
$\rho=\sigma(0,n-j,n)\circ \sigma(i',j',k')\in T_n.$ Also, $\rho\in B$ as $\rho_1\neq 1$ and $\rho_n\neq n$. Since $B$ together with the identity is a group, $\sigma(0,j,n)\circ\rho$ is also in $B$. This yields $\sigma(i',j',k')\in B,$ a contradiction with Lemma~\ref{lem1oct9}.
\end{proof}

The proofs of the subsequent properties use a few more equations involving block transpositions which are stated in the following two lemmas.

\begin{lemma}\label{lem2oct13}
In each of the following cases $\{\sigma(i,j,k),\sigma(i',j',k')\}$ is an edge of $\Cay.$
\begin{itemize}
\item[\rm(i)]$(i',j')=(i,j)$ and $k\ne k'$;
\item[\rm(ii)]$(i',j')=(j,k)$ for $k<k';$
\item[\rm(iii)]$(j',k')=(j,k)$ and $i\ne i'$;
\item[\rm(iv)]$(j',k')=(i,j)$ for $i'<i;$
\item[\rm(v)]$(i',k')=(i,k)$ and $j\ne j'$.
\end{itemize}
\end{lemma}
\begin{proof}
(i) W.l.o.g.
$k'<k.$ By \eqref{eq22ott12}, $\sigma(i,j,k)=\sigma(i,j,k')\circ\sigma(k'-j+i,k',k).$
(iii) W.l.o.g.
$i'<i.$ From \eqref{eq22ott12},
$\sigma(i,j,k)=\sigma(i',j,k)\circ\sigma(i',k-j+i',k-j+i).$

In the remaining cases, from \eqref{eq22ott12},
$$\begin{array}{lll}
\sigma(i,j,k)=\sigma(j,k,k')\circ\sigma(i,k'-k+j,k'),\\
\sigma(i,j,k)=\sigma(i',i,j)\circ\sigma(i',j-i+i',k),\\
\sigma(i,j,k)=\sigma(i,j',k)\circ\sigma(i,k-j'+j,k),
\end{array}$$
where by the last equality we assume that $j<j'$. Hence the statements hold.
\end{proof}

\begin{lemma}\label{prop1oct12}
The following equations hold.
\begin{itemize}
\item[\rm(i)]$\sigma(i,j,n)=\sigma(0,j,n)\circ\sigma(0,n-j,n-j+i)$ for $i\neq 0;$
\item[\rm(ii)]$\sigma(i,j,n)=\sigma(0,i,j)\circ\sigma(0,j-i,n)$ for $i\neq 0;$
\item[\rm(iii)]$\sigma(0,j,n)=\sigma(i,j,n)\circ\sigma(0,i,n-j+i)$
for $i\neq 0;$
\item[\rm(iv)]$\sigma(0,j,n)=\sigma(0,j,j+i)\circ\sigma(i,j+i,n)$ for
$0< i < n-j.$
\end{itemize}
\end{lemma}

\begin{proof}
(i),(ii) and (iv) are particular cases of the equations derived
in the previous proof. Namely,  if $i'< i,$ we have shown that 
$\sigma(i,j,k)=\sigma(i',j,k)\circ\sigma(i',k-j+i',k-j+i)=
\sigma(i',i,j)\circ\sigma(i',j-i+i',k)$.
Now, (i) and (ii) follow by substituting $k=n$ and $i'=0$ in
the first and the second equation, respectively. Also, if  $k'<k,$
then $\sigma(i,j,k)=\sigma(i,j,k')\circ\sigma(k'-j+i,k',k)$.
This yields (iv) by substituting at first $i=0$ and $k=n,$ and then
letting $k'=i+j$.
Finally,  we obtain (iii) by multiply both sides of the equation in (i) for 
$\sigma(0,n-j,n-j+i)^{-1}=\sigma(0,i,n-j+i)$.
\end{proof}

\begin{lemma}
\label{leaoct19}
Let $i$ be an integer with $0<i\leq n-2.$
\begin{itemize}
\item[\rm(i)] If $\sigma(i,j,n)=\sigma(0,\bar{j},n)\circ\sigma(i',j',k'),$ then $\bar{j}=j.$
\item[\rm(ii)] If $\sigma(i,j,n)=\sigma(i',j',k')\circ \sigma(0,\bar{j},n),$ then $\bar{j}=j-i.$ 
\end{itemize}
\end{lemma}
\begin{proof} (i) Assume $\bar{j}\neq j.$ From Lemma~\ref{prop1oct12} (i) and \eqref{eqa18oct},
\begin{equation}\label{matoct25}
\sigma(i',j',k')=\sigma(0,j^*,n)\circ\sigma(0,n-j,n-j+i),
\end{equation}where $j^*$ denotes the smallest positive integer such that $j^*\equiv j-\bar{j} \pmod n$. First we prove $i'=0$. Suppose on the contrary, then $$(\sigma(0,j^*,n)\circ\sigma(0,n-j,n-j+i))_1=1.$$On the other hand, $\sigma(0,n-j,n-j+i)_1=n-j+1$ and $\sigma(0,j^*,n)_{n-j+1}=n-\bar{j}+1$ since
$\sigma(0,j^*,n)_t=t+j^*\pmod n$ by \eqref{feb9}. Thus, $n-\bar{j}+1=1,$ a contradiction since $\bar{j}<n.$

Now, from \eqref{matoct25}, $\sigma(0,j',k')_n\neq n.$ Hence $k'=n$. Therefore, $$\sigma(0,n-j,n-j+i)=\sigma(0, n-j^*
,n)\circ \sigma(0,j',n)\in B.$$A contradiction since $i\neq j$.
This proves the assertion.

(ii) Taking the inverse of both sides of the equation in (ii) gives by \eqref{eqa18oct}
$$\sigma(i,n-j+i,n)=\sigma(0,n-\bar{j},n)\circ \sigma(i',j',k')^{-1}.$$ Now, from (i), $n-\bar{j}=n-j+i$, and the assertion follows.
\end{proof}

\begin{proposition}\label{prova}
The bipartite graphs arising from the components of the partition in Lemma~\ref{lem1oct9} have the following properties.
\begin{itemize}
\item[\rm(i)] In the bipartite subgraph $(L\cup F,B)$ of $\Cay,$ every vertex in $L\cup F$ has degree $1$ while every vertex of $B$ has degree $n-2.$
\item[\rm(ii)] The bipartite subgraph $(L,F)$ of $\Cay$ is a $(1,1)$-biregular graph.
\end{itemize}
\end{proposition}
\begin{proof} (i) Lemma~\ref{leaoct19} (i) together with Lemma~\ref{prop1oct12} (i) show that every vertex in $F$ has degree $1.$ Corollary~\ref{lem2oct11} ensures
that this holds true for $L.$

For every $1\le j \le n-1,$ Lemma~\ref{prop1oct12} (iii) shows that there exist at least $j-1$ edges incident with $\sigma(0,j,n)$ and a vertex in $F.$ Furthermore, from Lemma~\ref{prop1oct12} (iv), there exist at least $n-j-1$ edges incident with $\sigma(0,j,n)$ and a vertex in $L$. Therefore, at least $n-2$ edges incident with $\sigma(0,j,n)$ have a vertex in $L\cup F$. On the other hand, this number cannot exceed $n-2$ since $|L\cup F|=(n-1)(n-2)$ from Lemma~\ref{lem1oct9}. This proves the first assertion.

(ii) From Lemma~\ref{prop1oct12} (ii), there exists at least one edge with a vertex in $F$ and another in $L$. Also, Lemma~\ref{leaoct19} (ii) ensures the uniqueness of such an edge.
\end{proof}

From now on, $\Gamma(W)$ stays for the induced subgraph of $\Gamma$ on the vertex set $W.$

 \begin{corollary}
\label{cor3oct13ter}
$B$ is the unique maximal clique of $\Gamma$ of size $n-1$ containing an edge of $\Gamma(B).$
\end{corollary}
\begin{proof}
Proposition~\ref{prova} (i) together with
Lemma~\ref{lem1oct11}
show that the endpoints of an edge of $\Gamma(B)$ do not have a common neighbor outside $B.$
\end{proof}

Computations performed by using the package ``grape'' of GAP~\cite{gap} show that $\Gamma$ is a $6$-regular subgraph for $n=5$ and $8$-regular subgraph for $n=6,$ but $\Gamma$ is only $3$-regular for $n=4.$ This generalizes to the following result.

\begin{proposition}\label{thm1oct11}
$\Gamma$ is a $2(n-2)$-regular graph whenever $n\geq 5.$
\end{proposition}
\begin{proof}
Since $B$ is a maximal clique of size $n-1$, every vertex of $B$ is incident with $n-2$ edges of $\Gamma(B)$ and has no neighbor in $S_{n-2}^\vartriangle,$ by Lemma~\ref{lem1oct11}.
From Proposition~\ref{prova} (i), as many as $n-2$ edges incident with a vertex in $B$ have an endpoint in $L\cup F,$ Thus, the assertion holds for the vertices in $B.$

In $\Gamma(F\cup S_{n-2}^\vartriangle)$ every vertex has degree $2(n-1)-4=2n-6,$ by induction on $n.$
This together with Proposition~\ref{prova} (ii) show that every vertex of $\Gamma(F)$ has degree $2n-5$ in $\Gamma(L\cup F\cup S_{n-2}^\vartriangle).$ By Corollary~\ref{lem2oct11}, this holds true for every vertex of $\Gamma(L).$ The degree increases to $2n-4$ when we also count the unique neighbor
in $\Gamma(B),$ according to the first assertion of Proposition~\ref{prova} (i).

In $\Gamma(S_{n-2}^\vartriangle)$ every vertex has degree $2n-8,$ by induction on $n.$  
Since $\Gamma(F\cup S_{n-2}^\vartriangle)$ has degree $2n-6$, then $\Gamma(L\cup S_{n-2}^\vartriangle)$ has degree $2n-6$, by Corollary~\ref{lem2oct11}. This together with Lemma~\ref{lem1oct11} show that every vertex in $S_{n-2}^\vartriangle$ is the endpoint of exactly $2(2n-6)-(2n-8)$ edges in $\Gamma.$
\end{proof}

Our next step is to determine the set of all maximal cliques of
$\Gamma$ of size $2.$ From now on, we will be referring to the edges of the complete graph arising from a clique as the edges of the clique.
According to Lemma~\ref{lem2oct13} (v), let $\Lambda$ be the set of all edges $$e_l=\{\sigma(l,l+1,l+3),\sigma(l,l+2,l+3)\},$$ where $l$ ranges over $\{0,1,\ldots n-3\}$. From \eqref{eqa18oct}, the endpoints of such an edge are the inverse of one another.

\begin{proposition}\label{thm2oct11}
Let $n\geq 5.$ The edges in $\Lambda$ together with three more
edges
\begin{equation}
\label{eq1aoct12}
\begin{array}{lll}
e_{n-2}&=&\{\sigma(0,n-2,n-1),\sigma(0,n-2,n)\};\\
e_{n-1}&=&\{\sigma(1,n-1,n),\sigma(0,1,n-1)\};\\
e_{n}&=&\{\sigma(0,2,n),\sigma(1,2,n)\};\\
\end{array}
\end{equation}are pairwise disjoint edges of maximal cliques of
$\Gamma$ of size $2.$
\end{proposition}
\begin{proof} Since $n\geq 5,$ the above edges are pairwise disjoint.

Now, by \eqref{22march2015}, the following equations
\begin{equation}\label{eq1oct13}
\begin{array}{llllll}
\bar{\f}(\sigma(l,l+1,l+3))&=&\sigma(l-1,l,l+2)& \mbox{ for } l\geq 1;\\
\bar{\f}(\sigma(l,l+2,l+3))&=&\sigma(l-1,l+1,l+2) & \mbox{ for } l\geq 1; \\
\bar{\f}(\sigma(0,1,3))&=&\sigma(0,2,n);\\
\bar{\f}(\sigma(0,2,3))&=&\sigma(1,2,n);\\
\bar{\f}(\sigma(0,2,n)&=&\sigma(1,n-1,n);\\
\bar{\f}(\sigma(1,2,n))&=&\sigma(0,1,n-1);\\
\bar{\f}(\sigma(1,n-1,n)&=&\sigma(0,n-2,n-1); \\
\bar{\f}(\sigma(0,1,n-1))&=&\sigma(0,n-2,n);\\
\bar{\f}(\sigma(0,n-2,n-1))&=&\sigma(n-3,n-2,n); \\
\bar{\f}(\sigma(0,n-2,n))&=&\sigma(n-3,n-1,n).
\end{array}
\end{equation}
hold. This shows that $\bar{\f}$ leaves the set $\Lambda\cup\{e_{n-2},e_{n-1},e_n\}$ invariant acting on it as the cycle permutation
$(e_n,\,e_{n-1},\cdots, e_1,\,e_0).$

Now, it suffices to verify that $e_n$ is a maximal clique of $\Gamma.$ Assume on the contrary that $\sigma=\sigma(i,j,k)$ is adjacent to both $\sigma(1,2,n)$ and $\sigma(0,2,n).$ As $\sigma(0,2,n)\in B,$ Lemma~\ref{lem1oct11} and Proposition~\ref{prova} (i) imply
that $\sigma\in L\cup F$. Also, Proposition~\ref{prova} (i) shows that $\sigma(0,2,n)$ has degree $n-2$ in $L\cup F$. In particular, in the proof of Proposition~\ref{prova} (i), we have seen that $\sigma(0,2,n)$ must be adjacent to $n-3$ vertices of $L,$ as $\sigma(1,2,n)\in F.$ Then, by Lemma~\ref{prop1oct12} (iv), $\sigma=\sigma(0,2,l)$ for some $l$ with $3\leq l < n.$

On the other hand, Proposition~\ref{prova} (ii) shows that $\sigma\in L$ is uniquely determined by $\sigma(1,2,n)\in F,$ and, by Lemma~\ref{prop1oct12} (ii), $\sigma=\sigma(0,1,2),$ a contradiction.
\end{proof}

From now on, $V$ denotes the set of the vertices of the edges $e_m$ with $m$ ranging over $\{0,1,\ldots,n\}.$ For $n=4,$ the edges $e_m$ are not pairwise disjoint, but computations show that they are also edges of maximal cliques of $\Gamma$ of size $2.$

\begin{lemma}\label{lem1oct18}
The toric maps and the reverse map preserve $V.$
 If $n \ge 5,$ then
the toric-reverse group is regular on $V,$ and $\Gamma(V)$ is a vertex-transitive graph.
\end{lemma}
\begin{proof}
Since $\bar{\F}$ is the subgroup generated by $\bar{\f}$, from \eqref{eq1oct13} follows that $\bar{\F}$ preserves $V$ and has two orbits on $V,$ each of them containing one of the two endpoints of the edges $e_m$ with $0\le m \le n.$

In addition, by \eqref{eq1oct11}, the reverse map $\g$ interchanges
the endpoints of $e_m$ with $0\le m\le n-3$ and $m=n-1$ while
$$\begin{array}{lll}
\g(\sigma(0,n-2,n-1))&=&\sigma(1,2,n);\\
\g(\sigma(0,n-2,n)&=&\sigma(0,2,n);\\
\g(\sigma(1,2,n))&=&\sigma(0,n-2,n-1);\\
\g(\sigma(0,2,n))&=&\sigma(0,n-2,n).
\end{array}$$This implies that $\g$ preserves $V$, and
$\D_{n+1}$ acts transitively on $V.$

Now, since $|V|=2(n+1)$ and $\D_{n+1}$ has order $2(n+1)$, then
$\D_{n+1}$ is regular on $V.$
\end{proof}

In the action of $\D_{n+1}$ on $\Sym_n,$
the size of any $\D_{n+1}$-orbit $X$ is a divisor of $2(n+1),$ by the Orbit-Stabilizer Theorem. $X$ is a
\emph{long $\D_{n+1}$-orbit} if $|X|=2(n+1)$.
Now, Lemma~\ref{lem1oct18}  has the following consequence.

\begin{corollary}\label{propB19apr2015}
$V$ is a long $\D_{n+1}$-orbit whenever $n\ge 5$.
\end{corollary}

Our next step is to show that the $e_m$ with $0\le m \le n$ are the edges of all maximal cliques of $\Gamma$ of size $2.$ Computations performed by using the package ``grape'' of GAP~\cite{gap} show that the assertion is true for $n =4,5,6.$

\begin{lemma}\label{cor1oct13}
The edge of an arbitrary
maximal clique of $\Gamma$ of size $2$ is one of the edges $e_m$ with $0\le m \le n.$
\end{lemma}
\begin{proof}
On the contrary take an edge $e$ of a maximal clique of $\Gamma$ of size $2$ other than the edges $e_m$.
 We proceed by induction on $n$ and assume
$n\ge 7$. Computations show that the lemma holds for $n=4,5,6$.
First, suppose that $e$ is an edge of
$\Gamma(T_{n-1})$. Then by induction on $n,$ $e$ must be one of
the following edges:
\begin{enumerate}[]
\item $\{\sigma(0,n-3,n-2),\sigma(0,n-3,n-1,)\};$
\item $\{\sigma(1,n-2,n-1),\sigma(0,1,n-2,)\};$
\item $\{\sigma(0,2,n-1),\sigma(1,2,n-1)\}.$
\end{enumerate}
Now, by Lemma~\ref{lem2oct13}, it is straightforward to check that these edges extend to a clique of
$\Gamma$ of size $3$ by adding the vertices
$\sigma(0,n-3,n), \sigma(1,n-2,n),$ and $\sigma(2,n-1,n).$
From Corollary~\ref{lem2oct11} and Lemma~\ref{lem1oct18}, it also follows
that $e$ cannot be an edge of $\Gamma(S_{n-1}^\triangledown)$.

Suppose $e$ connects a vertex in $F$ with
a vertex in $L$.  By Proposition~\ref{prova} (ii) and Lemma~\ref{prop1oct12} (ii), the edge $e$ is of the form $\{ \sigma(i,j,n),\sigma(0,i,j)\}$. Notice that $(i,j)\ne (1,n-1),$ as $e\ne e_{n-1}$. If $j \ne n-1,$ then the vertex
$\sigma(i,j,n-1)$ is adjacent to the endpoints of $e,$
see Lemma~\ref{lem2oct13}, a contradiction. Let $j=n-1$. Then $i > 1$.
Consider the image $\g(e)$. By \eqref{eq1oct11},  $\g(e)=\{ \sigma(0,n-j,n-i),\sigma(n-j,n-i,n)\}$. Now, $n-i < n-1,$ and we conclude that $\g(e)$ extends
to a clique of size $3$ by adding the vertex $\sigma(n-j,n-i,n-1),$ a
contradiction.

We are left with the case when $e$ has one endpoint in $B$.
Notice that the other endpoint of $e$ cannot be in $B,$ as $B$ is a
clique. Furthermore, by Corollary~\ref{lem2oct11}  and Lemma~\ref{lem1oct18}, we may assume that the other endpoint of $e$ is in $F$.

Now, let the endpoint of $e$ be in $B$ with $e=\sigma(0,j,n)$ for some $1\leq j\leq n-1.$ Then, by the proof of the first assertion of Proposition~\ref{prova} (i), the vertex $\sigma(0,j,n)$ is adjacent to $\sigma(\bar{i},j,n)$ for any $0\le \bar{i} < j.$
Since $\sigma(0,j,n)$ is adjacent to a vertex in $F$ and
$e \ne e_n,$  $j > 2$ holds.
As the vertices $\sigma(\bar{i},j,n)$ for $0\le \bar{i} < j$ are adjacent, by Lemma~\ref{lem2oct13} (iii), $e$ is an edge of the triangle of vertices $\sigma(0,j,n),\,\sigma(i',j,n),$ and $\sigma(\bar{i},j,n)$ with $i'\neq\bar{i}$ and $0 < 
i',\bar{i} < j,$ a contradiction.
\end{proof}

Lemma~\ref{cor1oct13} shows that $V$ consists of the endpoints of the edges of $\Gamma$ which are the edges of maximal cliques of size $2.$ Thus $\Gamma(V)$ is relevant for the study of $\Cay.$ We show some properties of $\Gamma(V).$

\begin{proposition}\label{lemgoct12}
$\Gamma(V)$ is a $3$-regular graph whenever $n\ge 5$.
\end{proposition}
\begin{proof} First, we prove the assertion for the endpoint $v=\sigma(0,2,n)$ of $e_{n}.$ By Lemma~\ref{lem2oct13} (i) (iii) (v), $\sigma(0,2,3),\sigma(1,2,n),$ and $\sigma(0,n-2,n)$ are neighbors of $v$. Since $\sigma(1,2,n)\in F$ and $\sigma(0,2,3)\in L,$ from the first assertion of Proposition~\ref{prova} (i), $v\in B$ is not adjacent to any other vertex in either $V\cap F$ or $V\cap L.$ Also, Lemma~\ref{lem1oct11} yields that no vertex in $V\cap S_{n-2}^\vartriangle$ is adjacent to $\sigma(0,2,n).$ Thus, $v$ has degree $3$ in $\Gamma(V).$

Now the claim follows from Lemma~\ref{lem1oct18}.
\end{proof}

\begin{remark} By a famous conjecture of Lov\'asz, every finite, connected, and  vertex-transitive graph contains a Hamiltonian cycle, except the five known counterexamples, see~\cite{Ba,L}.
Then, the second assertion of Lemma~\ref{lem1oct18} and Proposition~\ref{thm1oct13} show that the Lov\'asz conjecture holds for the graph $\Gamma(V).$
\end{remark}

\begin{proposition}\label{thm1oct13}
$\Gamma(V)$ is a Hamiltonian graph whenever $n\geq 5.$
\end{proposition}
\begin{proof}
Let $v_1=\sigma(n-4,n-3,n-1),\quad v_2=\sigma(n-4,n-2,n-1)$ be the endpoints of $e_{n-4}$. We start by exhibiting a path $\mathcal{P}$ in $V$ beginning with $\sigma(0,2,3)$ and ending with $v_1$ that visits all vertices $\sigma(l,l+1,l+3),\sigma(l,l+2,l+3)\in\Lambda$ with $0\leq l\leq n-4.$

For $n=5,\,v_1=\sigma(1,2,4),$ and $$\mathcal{P}=\sigma(0,2,3),\sigma(0,1,3),\sigma(1,3,4),v_1.$$

Assume $n>5.$ For every $l$ with $0\leq l\leq n-4,$ Lemma~\ref{lem2oct13} (ii) (v) show that both edges below are incident to $\sigma(l,l+1,l+3)$:
$$\{\sigma(l,l+1,l+3),\sigma(l+1,l+3,l+4)\},\quad\{\sigma(l,l+2,l+3),\sigma(l,l+1,l+3)\}.$$Therefore,
$$\begin{array}{ll}
\sigma(0,2,3),\sigma(0,1,3),\sigma(1,3,4),\ldots,\sigma(l,l+2,l+3),\sigma(l,l+1,l+3),\\\sigma(l+1,l+3,l+4),\ldots,\,v_1
\end{array}$$is a path $\mathcal{P}$ with the requested property.

By Lemma~\ref{lem2oct13}, there also exists a path $\mathcal{P'}$ beginning with $v_1$ and ending with $\sigma(0,2,3)$ which visits the other vertices of $V,$ namely
$$\begin{array}{lll}
v_1,\sigma(n-3,n-1,n),\sigma(n-3,n-2,n),\sigma(0,n-2,n),\sigma(0,n-2,n-1),\\\sigma(0,1,n-1),\sigma(1,n-1,n),\sigma(1,2,n),
\sigma(0,2,n),\sigma(0,2,3).
\end{array}$$
Since $n\ge 5,$
the vertices are all pairwise distinct.
Therefore, the union of $\mathcal{P}$ and $\mathcal{P'}$ is a cycle in $V$ that visits all vertices. This completes the proof.
\end{proof}

\begin{remark}
 By Proposition~\ref{lemgoct12} and Lemma~\ref{lem1oct18},
Proposition~\ref{thm1oct13} also follows from a result of Alspach and Zhang~\cite{AZ} who proved that all cubic Cayley graphs on dihedral groups have Hamilton cycles.
\end{remark}

\section{The proofs of Theorem~1 and Theorem~2}
\label{sec:6}

We are in a position to give a proof of Theorem~\ref{main1}. From Proposition~\ref{prop10oct}, the toric-reverse group
$\D_{n+1}$ induces a subgroup of $\Aut(\Gamma)$.
Also, $\textsf{D}_{n+1}$ is regular on $V,$ by the second assertion of Lemma~\ref{lem1oct18}. Therefore, Theorem~\ref{main1} is a corollary of the following lemma.

\begin{lemma}
\label{lemfoct12} The identity is the only automorphism of $\Gamma$ fixing a vertex of $V$ whenever $n\geq 5.$
\end{lemma}
\begin{proof} We prove the assertion by induction on $n.$ Computation shows that the assertion is true for $n=5,6.$ Therefore, we assume $n\geq 7.$

First, we prove that any automorphism of $\Gamma$ fixing a vertex $v\in V$ actually fixes all vertices in $V$.
Since $\D_{n+1}$ is regular on $V,$ we may
limit ourselves to take $\sigma(0,2,n)$ for $v.$ Let $H$ be the
subgroup of all elements in $\Aut(\Gamma)$ fixing
$\sigma(0,2,n).$

By Lemma~\ref{cor1oct13}, $H$ preserves
$\Gamma(V)$. Now, we look inside the action of $H$ on $\Gamma(V)$ and show that
$H$ fixes the edge $\{\sigma(0,2,n),\sigma(0,n-2,n)\}.$
By Proposition~\ref{lemgoct12}, $\Gamma(V)$ is $3$-regular. More precisely, the endpoints of the edges of $\Gamma(V)$ which are incident with $\sigma(0,2,n)$ are $\sigma(0,2,3),\,\sigma(1,2,n),$ and $\sigma(0,n-2,n),$ see Lemma~\ref{prop1oct12} (i) (iii) (v). Also, by Proposition~\ref{thm2oct11}, the edge $e_n
=\{\sigma(0,2,n),\sigma(1,2,n)\}$ is the edge of a maximal clique of $\Gamma$ of size $2$, and no two distinct edges of maximal cliques of $\Gamma$ of size $2$ have a common vertex. Thus, $H$ fixes $\sigma(1,2,n).$ Now, from Corollary~\ref{cor3oct13ter}, the edge $\{\sigma(0,2,n),\sigma(0,n-2,n)\}$ lies in a unique maximal clique of size $n-1$. By Lemma~\ref{lem2oct13} (i), the edge $\{\sigma(0,2,n),\sigma(0,2,3)\}$ lies on a clique of size $n-2$ whose set of vertices is $\{\sigma(0,2,k) \mid 3\le k\le n\}.$ Here, we prove that $C$ is a maximal clique.

By the first assertion of Proposition~\ref{prova} (i), $\sigma(0,2,3)$ is adjacent to a unique vertex in $B$, namely $\sigma(0,2,n).$ On the other hand, among the $2(n-2)$ neighbors of $\sigma(0,2,n)$ off $B$ exactly one vertex is off $C,$ namely $\sigma(0,1,2)$.
Since this vertex is not adjacent to $\sigma(0,2,3),$
we conclude that
$C$ does not extend to a clique of size $n-1$. Therefore, $H$ cannot interchange the edges
$\{\sigma(0,2,n),\sigma(0,n-2,n)\}$ and $\{\sigma(0,2,n),\sigma(0,2,3)\}$ but fixes both.

Also, by Proposition~\ref{lemgoct12} and Lemma~\ref{prop1oct12} (i) (iii), $\sigma(0,n-2,n)$ is adjacent to $\sigma(0,n-2,n-1)$ and $\sigma(n-3,n-2,n).$ Since $e_{n-2}$ is the edge of a maximal clique of $\Gamma$ of size $2$, $H$ fixes $e_{n-2}=\{\sigma(0,n-2,n-1),\sigma(0,n-2,n)\}$. This together with what we have proven so far shows that $H$ fixes $\sigma(n-3,n-2,n),$ and then the edge $e_{n-3}=\{\sigma(n-3,n-2,n),\sigma(n-3,n-1,n)\}.$

Now, as the edge $\{\sigma(0,2,n),\sigma(0,n-2,n)\}$ is in $\Gamma(B)$, Corollary~\ref{cor3oct13ter} implies that
$H$ preserves $B.$ Therefore, $H$ preserves also the
set of all vertices adjacent to a vertex in $B$, which implies that $H$
preserves $S_{n-2}^\vartriangle$ as well.
And, as $H$ fixes $\{\sigma(0,2,n),\sigma(0,2,3)\}$, $H$ must fix $\sigma(0,2,n)\in B$ and $\sigma(0,2,3)\notin B.$ Also, $e_0=\{\sigma(0,1,3),\sigma(0,2,3)\}$ is preserved by $H$, as we have seen above. Therefore,  $\sigma(0,1,3)$ is also fixed by $H.$ Furthermore, $\sigma(2,3,5)\in S_{n-2}^\vartriangle$ is adjacent to $\sigma(0,2,3)$ in $\Gamma(V)$, by Proposition~\ref{lemgoct12} and Lemma~\ref{prop1oct12} (ii); and then it is fixed by $H$, as $H$ preserves $S_{n-2}^\vartriangle$.
Therefore, we have that $H$ induces an automorphism group of $\Gamma(S_{n-2}^\vartriangle)$ fixing a vertex $\sigma(2,3,5)\in S_{n-2}^\vartriangle.$ Then $H$ fixes every block transpositions in $S_{n-2}^\vartriangle\cong T_{n-2},$ by the inductive hypothesis. In particular, $H$ fixes all the vertices in $V \cap S_{n-2}^\vartriangle,$ namely all vertices in $\Lambda$ belonging to $e_l$ with $0<l<n-3.$

This together with what proven so far shows that $H$ fixes all vertices of $V$ with only two possible exceptions, namely the endpoints of the edge $e_{n-1}=\{\sigma(1,n-1,n),\sigma(0,1,n-1)\}.$ In this exceptional case, $H$ would swap $\sigma(0,1,n-1)$ and $\sigma(1,n-1,n).$ Actually, this exception cannot occur since $\sigma(0,1,n-1)$ and $\sigma(1,n-1,n)$ do not have a common neighbor, and $H$ fixes their neighbors in $V.$ Therefore, $H$ fixes every vertex in $V$.
Hence, $H$ is the kernel of the permutation representation of
$\Aut(\Gamma)$ on $V$. Thus $H$ is a normal subgroup of $\Aut(\Gamma)$.

Our next step is to show that the block transpositions in $L\cup B$ are also fixed by $H$. Take any block transposition $\sigma(0,j,k).$ Then the toric class of $\sigma(0,j,k)$ contains a block transposition $\sigma(i',j',k')$ from $S_{n-2}^\vartriangle$. This is a consequence of the equations below which are obtained by using \eqref{22march2015}
\begin{equation}\label{eqa14oct}
\begin{array}{llll}
\bar{\f}^{\,2}(\sigma(0,j,k))&=&\sigma(j-2,k-2,n-1),& j\geq 3; \\
\bar{\f}^{\,3}(\sigma(0,1,k))&=&\sigma(k-3,n-2,n-1),&k\geq 4;\\
\bar{\f}^{\,4}(\sigma(0,1,2))&=&\sigma(n-3,n-2,n-1);&{}\\
\bar{\f}^{\,5}(\sigma(0,1,3))&=&\sigma(n-4,n-3,n-1);&{}\\
\bar{\f}^{\,4}(\sigma(0,2,k))&=&\sigma(k-4,n-3,n-1),& k\geq 5;\\
\bar{\f}^{\,5}(\sigma(0,2,3))&=&\sigma(n-4,n-2,n-1);&{}\\
\bar{\f}^{\,6}(\sigma(0,2,4))&=&\sigma(n-5,n-3,n-1).&{}\\
\end{array}
\end{equation}
Since $\sigma(i',j',k')\in S_{n-2}^\vartriangle,$ we know that
$H$ fixes $\sigma(i',j',k').$ From this we infer that $H$ also fixes $\sigma(0,j,k).$ In fact, as $\sigma(0,j,k)$ and $\sigma(i',j',k')$ are torically equivalent, $\bar{u}(\sigma(i',j',k'))=\sigma(0,j,k)$ for some
$\bar{u}\in\bar{\F}$. Take any $h\in H.$

As $H$ is a normal subgroup of $\Aut(\Gamma),$ 
there exists $h_1\in H$ such that $\bar{u}\circ h_1=h \circ \bar{u}.$
Hence
$$\sigma(0,j,k)=\bar{u}(\sigma(i',j',k'))=
(\bar{u}\circ h_1)
(\sigma(i',j',k'))=(h 
\circ\bar{u})(\sigma(i',j,',k')),$$
whence $\sigma(0,j,k)=h
(\sigma(0,j,k)).$ Therefore, $H$ fixes every block transposition in
$L\cup B.$

Also, this holds true for $F,$ by the second assertion of Proposition~\ref{prova}. Thus, by Lemma~\ref{lem1oct9}, $H$ fixes
every block transposition. This completes the proof.
\end{proof}

\begin{remark}
Lemma~\ref{lemfoct12} yields Theorem~\ref{main1} for $n\geq 5.$ For $n=4,$ computations performed by using the package ``grape'' of GAP~\cite{gap} show that Theorem~\ref{main1} is also true.
\end{remark}

As a corollary to Theorem~\ref{main1}, we obtain
\begin{equation}\label{eq14sept}
\Aut(\Cay)= L(\Sym_n)(N 
\rtimes \D_{n+1}),
\end{equation}
where $N$ denotes the kernel of
the stabilizer of $\iota$ in $\Aut(\Cay)$ acting on
$T_n=V(\Gamma)$. In fact, we are going to derive
Theorem~\ref{main2} by showing that $N$ is a trivial group.

The proof of Theorem~\ref{main2} requires a property of the set $V$ that involves the alternating group $\Alt_n$ on $[n]$.

\begin{lemma}\label{10apr2015}
Let  $H$ be the subgroup of $\Sym_n$ generated by the permutations in $V$. Then
 $H$ is either $\Sym_n$ or $\Alt_n$ according to
as $n$ is odd or even.
\end{lemma}
\begin{proof} By the definition of $V$, the block transpositions $\sigma_l=\sigma(l,l+1,l+3)$ with $l=0,\ldots,n-3$ are all in $V$. First, we prove by induction on $n$ that they generate $\Alt_n$. If $n=3$, then $\sigma_0=[2\,3\,1]$, and hence it has order $3$. Therefore $\langle\sigma_0\rangle=\Alt_3$.

Let $n\geq 4$. Since $\sigma_i$ with $0\le i \le n-4$ fixes $n$, we may assume by induction that $\langle \sigma_0,\sigma_1,\ldots,\sigma_{n-4}\rangle=\Alt_{n-1}$.
Since $\sigma_{n-3}=\sigma(n-3,n-2,n)=[1\,\cdots\,(n-1)\, n\, (n-2)],$ 
by \eqref{eq22ott12}, we have $\sigma_{n-3}\in\Alt_n$.
Furthermore $\sigma_{n-3}$ moves $n$. Therefore $\langle \sigma_0,\ldots,\sigma_{n-3}\rangle=\Alt_n$.

The remaining six block transpositions of $V$ are
$\sigma(0,1,n-1),
\,\sigma(1,2,n),$\\ $\sigma(0,2,n),$ and their inverses. By \eqref{eq22ott12},
both $\sigma(0,1,n-1)$ and $\sigma(1,2,n)$ are cycles
of length $n-1,$ and
$\sigma(0,2,n)=\sigma(0,1,n)^2.$ 
Then $\sigma(0,2,n)\in\Alt_n$, while
$\sigma(0,1,n),\,\sigma(0,2,n)\in \Alt_n$ if and only if $n$ is even. Therefore, $H=\Alt_n$ if $n$ is even, and
$H=\Sym_n$ otherwise.
This completes the proof.
\end{proof}

Before presenting our proof of Theorem~\ref{main2} we need
one more auxiliary lemma.

\begin{lemma}\label{L2}
Let $U$ be a symmetric $\D_{n+1}$-orbit contained in $T_n$, and let $t$ be an
automorphism of $\Cay$ fixing point-wise an edge
$\{k,l\}$ of $\UCay$. If $U$ is a long orbit, then $t(\pi) = \pi$ whenever $\pi$ lies on the same component of $k$ in $\UCay$.
\end{lemma}
\begin{proof}
For the sake of simplicity, let $A=\Aut(\Cay)$.
For $\pi\in\Sym_n,$ let $A_\pi$ denote the vertex stabilizer of $\pi$ in
$A;$ furthermore, let $\widehat{A}_\pi$ denote the subgroup of
$A_\pi$ consisting of those elements fixing also each neighbor of $\pi$ in $\UCay$, that is
$$\widehat{A}_\pi =\big\{ t \in A \, \mid \,
t(\rho)=\rho \text{ if } \rho=\pi \text{ or }
\rho \sim \pi \text{ in } \UCay \big\}.$$
Since $U$ is a long $\D_{n+1}$-orbit contained in $T_n$ and
$A_\iota$
acts on $T_n$ as $\D_{n+1}$, we have
\begin{equation}\label{hat}
\mbox{$h \in A_\iota,
\, h(\nu)= \nu$
and  $\nu \in U$ imply $h \in
\widehat{A}_{\iota}$}.
\end{equation}
Now, let us consider an automorphism $t\in A$
fixing both endpoints of the edge $\{k,l\}$ of $\UCay.$
Observe that 
$A_k = L_k\circ A_{\iota}\circ L_{k^{-1}}.$ We show that $t \in \widehat{A}_k$.

Let $h=L_{k^{-1}}\circ t\circ L_k,$ 
and let $\nu = k^{-1}\circ l$.
Thus, $h \in A_\iota, 
\,\nu \in U$, and $h(\nu)=\nu$. Hence
$h \in  \widehat{A}_{\iota},$ 
by \eqref{hat}. In particular, if $\rho$ is any neighbor of $k$, then $h(k^{-1}\circ\rho)=k^{-1}\circ\rho.$ Therefore,
$$t(\rho)=(L_k\circ h\circ L_{k^{-1}})(\rho)=k\circ h(k^{-1}\circ\rho)=k\circ(k^{-1}\circ \rho)=\rho$$ showing that
$t \in  \widehat{A}_{k}$. 
This also yields that $t(\pi)=\pi$
whenever there is a path from $k$ to $\pi$.
\end{proof}

Now, we are in a position to prove Theorem~\ref{main2}.
\medskip

\noindent{\it Proof of Theorem~\ref{main2}.}
We have to show that the kernel $N$ given in \eqref{eq14sept} is
trivial Assume on the contrary that $N$ is non-trivial, that is, some non-trivial element of 
$\Aut(\Cay)$ fixes every vertex in $T_n$. Choose such an element
$t$, and take for $U$ the set $V$, as in Lemma~\ref{L2}.
Also, choose an element $v\in U$. Then $\iota$ and $v$ are adjacent in $\UCay$, and Lemma~\ref{L2} applies to $k=\iota,\,l=v$. It turns out that $t$ fixes
not only all block transpositions, but also every permutation $\pi\in\Sym_n$ whenever $v$ and $\pi$ are in the same component in $\UCay$. The latter property means that $\pi \in H$, where
$H$ is the subgroup of $\Sym_n$ generated by $U$.

Then $t$ must fix every permutation in $\Alt_n$ and $n$ must be even, by Lemma~\ref{10apr2015}. Actually, $t$ also fixes
 two adjacent odd permutations, for instance $\sigma(0,1,2)$
and $\sigma(1,2,3)$. Therefore, $t$ fixes every permutation in $\Sym_n$, that is, $t$ is the trivial element of
$\Aut(\Cay)$, a contradiction. \hfill $\Box$

\section{Regular Cayley maps and the toric groups}
\label{sec:7}

Throughout this section, we adopt the notation used in \cite{CJT,KK}
and restrict our attention to Cayley maps with simple underlying
graphs.

In particular, let $\Cayg(G,X)$  be a connected Cayley graph
for a finite group $G,$ and let $p$ be a cyclic permutation of $X$. Then, the \emph{Cayley map} $\CM(G,X,p)$ is the $2$-cell
embedding of the graph $\Cayg(G,X)$ in an orientate surface, where   the orientation induced by a local ordering of the arcs emanating
from any vertex $g\in G$ is always the same as the order of
the generators in $X$ induced by $p$. 
In other words, if $X$ has $k$ elements, 
then the neighbors of any vertex $g$  are always spread counterclockwise around $g$ in
the order $(gx,gp(x),\ldots,gp^{k-1}(x))$.

For an \emph{arc} of $\Cayg(G,X)$ we consider an ordered pair
$(u,v)$ of vertices such that $\{u,v\}$ is an edge of $\Cayg(G,X)$.
And, we will also refer to the arcs of $\Cayg(G,X)$ as
the \emph{darts} of the Cayley map.

In order to define the faces,
we need two permutations of the dart
set, namely the \emph{rotation} $R$ which maps the dart
$(g,gx)$ to the dart $(g,gp(x))$ and the \emph{dart-reversing involution} $T$ which swaps $(g,xg)$ with $(xg,g)$. The product
$R\circ T$ divides the darts into certain orbits, and the
cycle sequences of vertices appearing in these orbits define the
\emph{faces} of $\CM(G,X,p)$.

\begin{example}\label{octahedron}
Let $\M=\CM(\Sym_3,X,p)$ be the Cayley map, where
\begin{eqnarray*}
X &=& \{\sigma(0,1,3),\sigma(0,2,3),\sigma(0,1,2),\sigma(1,2,3)\}{\color{blue};} \\
p &=& \big(\sigma(0,1,3),\sigma(0,2,3),\sigma(1,2,3),\sigma(0,1,2)\big).
\end{eqnarray*}
The $24$ darts of $\M$ are partitioned in
$8$ faces, and each of them forms a triangle.
The resulting map is the
well-known octahedron embedded in the sphere, see Fig.~1.
\end{example}

\begin{figure}[ht!]
\centering
\begin{tikzpicture}[thick,scale=4.5]
\coordinate (A1) at (0,0);
\coordinate (A2) at (0.6,0.2);
\coordinate (A3) at (1,0);
\coordinate (A4) at (0.4,-0.2);
\coordinate (B1) at (0.5,0.5);
\coordinate (B2) at (0.5,-0.5);

\begin{scope}[thick,dashed,,opacity=0.6]
\draw (A1) -- (A2) -- (A3);
\draw (B1) -- (A2) -- (B2);
\end{scope}
\draw (A1) -- (A4) -- (B1);
\draw (A1) -- (A4) -- (B2);
\draw (A3) -- (A4) -- (B1);
\draw (A3) -- (A4) -- (B2);
\draw (B1) -- (A1) -- (B2) -- (A3) --cycle;
\draw (B1) node[above]{$\iota=[1\,2\,3]$};
\draw (A1) node[left]{$\sigma(0,1,3)=[2\,3\,1]$};
\draw (A4) node[left]{$\sigma(0,2,3)=[3\,1\,2]$};
\draw (A3) node[right]{$\sigma(1,2,3)=[1\,3\,2]$};
\draw (A2) node[right]{$\sigma(0,1,2)=[2\,1\,3]$};
\draw (B2) node[below]{$[3\,2\,1]$};
\end{tikzpicture}
\caption{The octahedron as a regular Cayley map on $\Sym_3$.}
\end{figure}
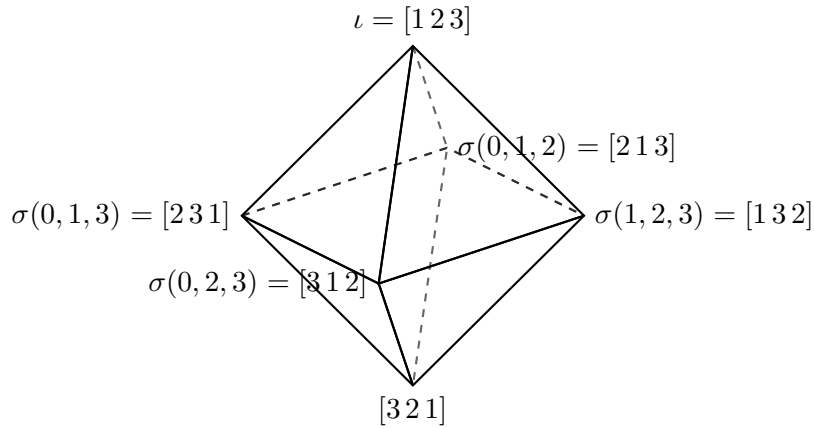

For an \emph{automorphism} of a Cayley map $\M=\CM(G,X,p)$ we refer to a bijective map $\varphi$ of the vertex set which preserves both the dart and the face sets, equivalently, both $\varphi\circ R=R\circ\varphi$ and $\varphi\circ T=T\circ\varphi$ hold, where $R$ and $T$
are the rotation and dart-reversing involution associated with $\M$.
Here, the group of all automorphisms of $\M$ is denoted by $\Aut(\M)$. This group always acts semiregularly on the darts of $\M$, and 
when the action is also transitive (and hence regular), we name
$\M$ a \emph{regular Cayley map}.

Jajcay and \v{S}ir\'a\v{n} \cite{JS} characterized regular Cayley
maps using skew-morphisms. For an arbitrary finite group $G,$ let $\psi$ be a permutation of $G$ of order $r,$ and let $\pi$ be any function from $G$ to $\{0,\ldots,r-1\}$.  The permutation $\psi$ is a
\emph{skew-morphism}  of $G$ with \emph{power function} $\pi$
if $\psi(1_G)=1_G$ and $\psi(xy)=\psi(x)\psi^{\pi(x)}(y)$ for
all $x,y \in G$ ($1_G$ denotes the identity element of $G$).
The relation between regular Cayley maps and skew-morphisms
is captured in the following theorem.

\begin{theorem}[Jajcay and \v{S}ir\'a\v{n} \cite{JS}]\label{JS}
A Cayley map $\CM(G,X,p)$ is regular if and only if there exists a
sew-morphisms $\psi$ of $G$ such that $p(x)=\psi(x)$ for
all $x\in X$.
\end{theorem}

The skew-morphism $\psi$ in the above theorem and its
power function are determined uniquely by the regular Cayley map $\CM(G,X,p).$  Therefore, we will refer to them as the skew-morphism and the power function associated with
$\CM(G,X,p)$. Also, the automorphism group of 
$\CM(G,X,p)$ satisfies 
$$ \Aut(\CM(G,X,p))=L(G)\langle\,\psi\,\rangle.$$ 

Following \cite{CJT}, we say that a regular 
Cayley map $\CM(G,X,p)$ is \emph{$t$-balanced} for some integer
$1\le t\le |X|-1$ if the associated power function $\pi$ satisfies
$\pi(x)=t$ for all  $x\in X$. In the extremal case when 
$t=1,$ the regular Cayley map is also called \emph{balanced},
and when $t=|X|-1$ it is called \emph{anti-balanced}.
\medskip

In the next proposition we describe a regular Cayley map on the
group $\Sym_n$ whose
underlying graph is a subgraph of $\Cayg(\Sym_n,T_n),$ and
whose automorphism group is equal to $L(\Sym_n)\bar{\F}$.
In fact, this Cayley map for $n=3$ coincides with the octahedron
given in Example~\ref{octahedron}.

\begin{proposition}\label{prop11sept}
Let $n\ge 3$ and $\M=\CM\big(\Sym_n,X,p)$ be the Cayley map, where
\begin{eqnarray*}
X &=& \{ \sigma(0,1,n),\sigma(0,n-1,n), \sigma(i,i+1,i+2) \mid i=0,\ldots,n-2 \}, \\
p &=& \big(\sigma(0,1,n),\sigma(0,n-1,n),\sigma(n-2,n-1,n),
\ldots,\sigma(0,1,2)\big).
\end{eqnarray*}
Then, $\M$ is a regular Cayley map which is non-$t$-balanced for
any $t.$ Also, it has valency $n+1,$ face size $n,$ and 
$\Aut(\M)=L(\Sym_n)\bar{\F}\cong\Sym_{n+1}$.
\end{proposition}

\begin{proof}
To show that $\M$ is regular, it suffices to observe that Lemma~\ref{mat27oct} is equivalent to say that the toric map
$\bar{\f}$ is a skew-morphism of $\Sym_n$ whose power function
$\pi$ satisfies $\pi(\rho)=(\rho^{-1})_1$ for all $\rho\in\Sym_n$.
By Lemma~\ref{22marchC2015}, it is straightforward to check that 
$\bar{\f}(x)=p(x)$ for all $x\in X,$  and hence $\M$ is indeed regular according to Theorem~\ref{JS}.
The automorphism group $\Aut(\M)=L(\Sym_n)\langle\,\bar{\f}\,\rangle=L(\Sym_n)\bar{\F}$ which is isomorphic to
$\Sym_{n+1},$ see Proposition~\ref{29oct}. Also,
$\pi(\sigma(0,1,n))=n$ and $\pi(\sigma(1,2,3))=1,$ thus $\M$ is
not $t$-balanced for any possible $t$.

Finally, since $\M$ is regular, every face of $\M$ has the same size.
Here, we compute the face incident with the dart
$(\iota,\sigma(0,1,n)).$ For this purpose let $\sigma=\sigma(0,1,n)$.
Let $R$ and $T$ be the rotation and dart-reversing involution associated with $\M$. A direct computation gives that
the orbit of $(\iota,\sigma)$ under $R\circ T$ consists of $n$ darts
with the induced closed walk:
$$\iota \; - \; \sigma \; - \;  \sigma\circ\sigma(n-2,n-1,n) \; - \;
\sigma\circ\sigma(n-2,n-1,n)\cdots\sigma(1,2,3).
$$
This shows that the face incident with the dart $(\iota,\sigma)$ indeed consist of $n$ vertices.
\end{proof}

As a main result of this section, we prove the following converse of
Proposition~\ref{prop11sept}.

\begin{theorem}\label{thm11sept}
Let $n\ge 5$ and $\M=\CM(\Sym_n,X,p)$ be a regular
non-$t$-balanced Cayley map, then $|X| \ge n+1.$ In particular, if
$|X|=n+1,$ then there exists a regular Cayley map $\M'$ on $\Sym_n$
isomorphic to $\M,$ with $\Aut(\M')=L(\Sym_n)\bar{\F}$.
\end{theorem}

\begin{proof}
Consider $\Aut(\M)$ acting on the left $L(\Sym_n)$-cosets in
$\Aut(\M)$ by left multiplication. More precisely, for
arbitrary automorphisms $a,b\in\Aut(\M),$
the coset $a\circ L(\Sym_n)$ is mapped by $b$ to the coset $b\circ a\circ L(\Sym_n)$. Let $K$ denote the kernel of this action. It is well-known that $K$ is equal to the largest normal subgroup of $\Aut(\M)$ contained in $\Sym_n$. Since $n \ge 5,$ if $K$ is non-trivial, then
$K=L(\Sym_n)$ or $L(\Alt_n),$ where
$L(\Alt_n)=\{L_\pi\mid\pi\in\Alt_n\}$. Here, we prove that actually those two cases do not occur. In fact, we find that in both cases $\M$ is $t$-balanced for a suitable integer $t$.

Let $\psi$ and $\pi$ be the skew-morphism and power function associated with $\M$. Let $k\in\Sym_n$ such that $L_k\in K$.
Since $K$ is normalized by $\psi,$ there exists $k_1\in\Sym_n$
such that $L_{k_1}\in K$  and $\psi\circ L_k=L_{k_1}\circ\psi$. Then $k_1=(L_{k_1}\circ\psi)(\iota)=
(\psi\circ L_k)(\iota)=\psi(k)$. Thus for every $\rho\in\Sym_n,$
$\psi(k\circ\rho)=(\psi\circ L_k)(\rho)=(L_{k_1}\circ\psi)(\rho)=
k_1\circ\psi(\rho)=\psi(k)\circ\psi(\rho)$.
On the other hand, $\psi(k\circ\rho)=\psi(k)\circ\psi^{\pi(k)}(\rho),$
and it follows that $\pi(k)=1$.
In particular, if $K=L(\Sym_n),$ then $\pi(x)=1$ for all $x\in X,$ and
$\M$ is balanced. Suppose that $K=L(\Alt_n)$.
It was shown in \cite{JS} that for all $\rho,\sigma\in\Sym_n,$
$\pi(\rho)=\pi(\sigma)$ if and only if $\pi(\rho\circ\sigma^{-1})=1$.
It follows from this that $\pi(\rho)=t$ for some integer $t\ge 1$ whenever $\rho$ is an odd permutation. Since $K=L(\Alt_n)$ is
normalized by $\psi$ and the identity $\iota$ is fixed by $\psi,$
$\psi$  maps $\Alt_n$ to itself.  Therefore, as $X$ is an
$\langle\,\psi\,\rangle$-orbit, see Theorem~\ref{JS}, the condition
$X\cap\Alt_n\ne\emptyset$ implies that $X\subset\Alt_n,$ which, however, is in contradiction with the fact that $X$ generates $\Sym_n$.
Thus $X\cap\Alt_n=\emptyset,$ and hence $\pi(x)=t$ for all $x\in X,$
and so $\M$ is $t$-balanced, as claimed.

We are left with the case when $K$ is trivial. In this case the
point stabilizer of the permutation group induced by the above action is
isomorphic to $\Sym_n$. Since it has degree
$|\Aut(\M) : L(\Sym_n)|=|\langle\,\psi\,\rangle|=|X|,\,(|X|-1)! \ge n!,$ and hence $|X| \ge n+1$. Moreover,
$|X|=n+1$ if and only if $\Aut(\M)\cong\Sym_{n+1}$. 
Now, it is more convenient to represent the symmetric group on $n+1$ elements as $\Sym_{[n]^0}$ since we are going to use two results about Cayley maps with a prescribed automorphism group proved in \cite[Section~3]{KK}.
By \cite[Lemma~3.3]{KK}, $\M$ is isomorphic to a Cayley
map $\CM(\Sym_{[n]^0},H,x,y)$  for a subgroup
$H\le\Sym_{[n]^0}$ and two permutations $x,y\in\Sym_{[n]^0}$
satisfying the following:
\begin{itemize}
\item $H\cong\Sym_n,$ $\Sym_{[n]^0}=H\circ Y$ and
$|H\cap Y|=1$ where $Y=\langle\, y\,\rangle;$
\item No non-trivial normal subgroup of $\Sym_{[n]^0}$ is contained
in $Y;$
\item   $\Sym_{[n]^0}=\big\langle Y,x \big\rangle$ and
$Y \circ x \circ Y=Y \circ x^{-1}\circ Y$.
\end{itemize}
By $\CM(\Sym_{[n]^0},H,x,y)$, we represent the Cayley map
$\CM(H,X,p),$ where
\begin{eqnarray*}
X &=& \{x_1,\ldots,x_{|Y|}\}, \\
p &=& (x_1,\ldots,x_{|Y|}),
\end{eqnarray*}
and the elements $x_i$ are defined recursively as $x_1=x.$ For $i > 1,$ $x_i$ is the
unique element of $H$ satisfying $x_i \circ Y=y\circ x_{i-1}\circ Y$.
In fact, the skew-morphism associated with
$\CM(\Sym_{[n]^0},H,x,y)$ maps an arbitrary permutation
$\rho\in H$ to the unique element $\rho'\in H$ for which
$\rho'\circ Y=y\circ\rho\circ Y$.

Here, we look for $H$ among the subgroups of
$\Sym_{[n]^0}$ isomorphic to $\Sym_n$. It is known that,
if $n\ne 6,$ then $H$ is conjugate to $\Sym_n$ in
$\Sym_{[n]^0},$ see \cite[5.5~Satz.~a)]{H}.
Now, we consider $\Sym_n$ as a subgroup of
$\Sym_{[n]^0}$. 
If $n=6,$ then $\Sym_n$ contains exactly 
two conjugacy classes of subgroups isomorphic to $\Sym_5,$
see \cite[5.5~Satz.~b)]{H}. Also, these two classes are switched by an automorphism of $\Sym_6$ (see, e.g.~\cite{M}). Therefore, by \cite[Lemma~3.1]{KK}, $H=\Sym_n$, up to the isomorphism of $\M.$

Now, the group $\Sym_{[n]^0}$ factorizes as
$\Sym_{[n]^0}=\Sym_n\circ Y$ with $|\Sym_n \cap Y|=1$.
These yield $|Y|=n+1,$ and that no non-identity
element of $Y$ can fix $0 \in [n]^0$. We conclude that $y$ is a cycle
in $\Sym_{[n]^0}$ of length $n+1$. Thus $y$ and the permutation
$\alpha=[1\,2\,\ldots\,n\,0]$ are conjugate via a suitable permutation
from $\Sym_n$.  Again, by \cite[Lemma~3.1]{KK},
$\M$ is isomorphic to a Cayley map
$\M'$ in the form
\begin{equation}\label{mapM'}
\M'=\CM(\Sym_{[n]^0},\Sym_n,x,\alpha).
\end{equation}

We conclude the proof by showing that the skew-morphism associated
with $\M'$ generates the toric group $\bar{\F},$ and thus
$\Aut(\M')=L(\Sym_n)\bar{\F},$ as claimed.  Let $\psi$ denote
the latter skew-morphism. An arbitrary permutation $\rho\in\Sym_n$ is mapped by $\psi$ to $\rho'$  defined by $\rho'\circ Y=\alpha\circ\rho\circ Y$.
Thus $\rho'=\alpha\circ\rho\circ\alpha^{s}$ for a suitable exponent
$s$. Then $0=\rho'_0=(\alpha\circ\rho\circ\alpha^{s})_0=
\alpha_{\rho_s},$ $\rho_s=n,$ and $s=(\rho^{-1})_n$. Comparing this with \eqref{eq2jul27}, we conclude that $\psi=\bar{\f}_n,$ and thus $\psi$ indeed generates $\bar{\F}$. This completes the proof of the theorem.
\end{proof}

\begin{remark}
The regular Cayley map
$\M'$ given in \eqref{mapM'} is not uniquely determined.
For instance, if $x=\sigma(0,1,n),$ then $\M'$ becomes the Cayley
map given in Proposition~\ref{prop11sept}. On the other hand,
when $n=5$ and $x=[5\,4\,2\,3\,1],$ $\M'$ becomes the Cayley map
$\CM(\Sym_5,X',p'),$ where
\begin{eqnarray*}
X' &=& \{ [5\,4\,2\,3\,1], 
[5\,3\,4\,2\,1],
[4\,5\,3\,2\,1],
[4\,3\,2\,1\,5],
[1\,5\,4\,3\,2],
[5\,4\,3\,1\,2]
\};  \\
p' &=& \big(
[5\,4\,2\,3\,1], 
[5\,3\,4\,2\,1],
[4\,5\,3\,2\,1],
[4\,3\,2\,1\,5],
[1\,5\,4\,3\,2],
[5\,4\,3\,1\,2]
\big).
\end{eqnarray*}
\end{remark}
\noindent 
Computations carried out with the package ``grape'' of GAP \cite{gap} showed that $\Cayg(\Sym_5,X')$ is
not isomorphic to $\Cayg(\Sym_5,X),$ where $X$
is the subset of $T_5$ given in Proposition~\ref{prop11sept}.
Consequently, also the Cayley maps $\CM(\Sym_5,X,p)$ and
$\CM(\Sym_5,X',p')$ are not isomorphic, whereas both of them have the same automorphism group, namely $L(\Sym_5)\bar{\F}$.

\end{document}